\documentclass[11pt,a4paper]{amsart}
\usepackage{amsthm}
\usepackage{amssymb}
\usepackage{latexsym}
\usepackage{amscd}
\usepackage{graphics}

\newcommand{\bprop} {\begin{proposition}}
\newcommand{\eprop} {\end{proposition}}
\newcommand{\btheo} {\begin{theorem}}
\newcommand{\etheo} {\end{theorem}}
\newcommand{\blem} {\begin{lemma}}
\newcommand{\elem} {\end{lemma}}
\newcommand{\bcor} {\begin{corollary}}
\newcommand{\ecor} {\end{corollary}}

\newcommand{\Be}{\begin{equation}}
\newcommand{\Ee}{\end{equation}}
\newcommand{\Bea}{\begin{eqnarray}}
\newcommand{\Eea}{\end{eqnarray}}
\newcommand{\Bes}{\begin{equation*}}
\newcommand{\Ees}{\end{equation*}}
\newcommand{\Beas}{\begin{eqnarray*}}
\newcommand{\Eeas}{\end{eqnarray*}}
\newcommand{\Ba}{\begin{array}}
\newcommand{\Ea}{\end{array}}

\newtheorem{theorem}{T{\hskip 0pt\footnotesize\bf HEOREM}}[section]
\newtheorem{lemma}[theorem]{L{\hskip 0pt\footnotesize\bf EMMA}}
\newtheorem{proposition}[theorem]{P{\hskip 0pt\footnotesize\bf ROPOSITION}}

\newtheorem{corollary}[theorem]{C{\hskip 0pt\footnotesize\bf OROLLARY}}

\begin{document}


\title{Bounded and invertible Toeplitz products on vector weighted Bergman spaces of the polydisc}

\author{Beno\^it F. Sehba}
\address{Department of Mathematics University of Ghana PO. Box LG 62 Legon Accra Ghana}


\email{bfsehba@ug.edu.gh}


\begin{abstract}
 We characterize bounded and invertible Toeplitz products on vector weighted Bergman spaces of the unit polydisc. For our purpose, we will need the notion of B\'ekoll\'e-Bonami weights in several parameters.
\end{abstract}



\maketitle

\section{Introduction}
Let us denote by $d\nu$ the normalized Lebesgue measure on the unit disc
$\mathbb D$ of $\mathbb C$.
For $\alpha>-1$, we denote by $d\nu_{\alpha}$ the normalized Lebesgue measure $d\nu_{\alpha}(z)=c_{\alpha}(1-|z|^2)^{\alpha}d\nu(z)$, $c_\alpha$ being the normalizing constant. The weighted Bergman space $A_\alpha^{p}(\mathbb D)$ is the space of holomorphic functions $f$ such that
$$||f||_{ A_\alpha^{p}}^p:=\int_{\mathbb D}|f(z)|^pd\nu_{\alpha}(z)<\infty.$$

%

The Bergman space $A_{\alpha}^{2}(\mathbb D)$ is a reproducing kernel Hilbert space with kernel $K_w^\alpha (z)=K^\alpha (w,z)=\frac{1}{(1-w\overline {z})^{2+\alpha}}$. That is for any $f\in A_{\alpha}^{2}(\mathbb D)$, the following representation holds
\begin{equation}
f(w)=P_{\alpha}f(w)=\langle f,K_w^{\alpha}\rangle_{\alpha}=\int_{\mathbb {D}}f(z)K^{\alpha}(w,z)d\nu_{\alpha}(z),\,\,\,w\in \mathbb D.
\end{equation}
For $f\in L^2(\mathbb {D}, d\nu_{\alpha})$, we can densely define the Toeplitz operator $T_f$ with symbol $f$ on $A_{\alpha}^{2}(\mathbb D)$ as follows
\begin{equation}\label{eq:toeplitzdef}
T_f(g)=P_{\alpha}(M_f)(g)=P_{\alpha}(fg)
\end{equation}
where $M_f$ is the multiplication operator by $f$.
The Berezin transform is the operator defined on $L^1(\mathbb {D}, d\nu_\alpha)$ by
$$B_\alpha (f)(w)=\int_{\mathbb {D}}f(z)|k_w^\alpha (z)|^2dv_\alpha(z)$$
where $k_w^\alpha$ is the normalized reproducing kernel of $A_\alpha^2(\mathbb {D})$.

 Invertible and bounded Toeplitz products have been first obtained for the Hilbert-Bergman space $A_\alpha^2(\mathbb D)$ by K. Stroethoff and D. Zheng in \cite{SZ1}. They proved the following.
\begin{theorem}[K. Stroethoff and D. Zheng \cite{SZ1}]\label{theo:invertp=2}
Given $f,g\in A_\alpha^2(\mathbb D)$, $-1< \alpha<\infty$, the product $T_fT_{\overline g}$ is bounded and invertible on $A_\alpha^2(\mathbb D)$ if and only if
\begin{equation}\label{eq:invertdisc1}
\inf_{w\in \mathbb D}|f(w)||g(w)|>0
\end{equation}
and \begin{equation}\label{eq:sarason}
\sup_{w\in \mathbb {D}}\left(B_\alpha(|f|^2)(w)\right)^{1/2}\left(B_\alpha(|g|^2)(w)\right)^{1/2}=\sup_{w\in \mathbb {D}}\|fk_w^\alpha\|_{2,\alpha}\|gk_w^\alpha\|_{2,\alpha}<\infty.
\end{equation}

\end{theorem}

The above result was recently extended to the unit polydisc  by Z. Sun and Y. Lu in \cite{LS}. In \cite{Isral}, J. Isralowitz extended Theorem \ref{theo:invertp=2} to all $A_\alpha^p(\mathbb D)$ using among others a result of J. Miao \cite{Miao} and B\'ekoll\'e-Bonami weights. More precisely, he proved the following result.
\begin{theorem}[J. Isralowitz \cite{Isral}]\label{theo:invertp}
Let $1<p,q<\infty$, $p=q(p-1)$ and $-1< \alpha<\infty$. Given $f\in A_\alpha^p(\mathbb D)$ and $g\in A_\alpha^q(\mathbb D)$,  the product $T_fT_{\overline g}$ is bounded and invertible on $A_\alpha^p(\mathbb D)$ if and only if
\begin{equation}\label{eq:invertdisc2}
\inf_{w\in \mathbb D}|f(w)||g(w)|>0
\end{equation}
and
\begin{equation}\label{eq:sarasonp}
\sup_{w\in \mathbb {D}}B_\alpha(|f\left(k_w^\alpha\right)^{1-\frac{2}{p}}|^p)(w)B_\alpha(|g\left(k_w^\alpha\right)^{1-\frac{2}{p}}|^p)(w)<\infty.
\end{equation}
\end{theorem}
Our aim in this paper is to extend the above result to the polydisc. Our motivation comes from the current and growing research activity on weighted norm inequalities for the Bergman projection and their application to some other operators as Toeplitz products.

\section{Statement of the result}

We recall that $\mathbb D^n$ is the unit polydisc, where $\mathbb D$ is the unit disc of $\mathbb C$. For any two points $z = (z_1 ,\cdots, z_n )$ and $w = (w_1 ,\cdots, w_n )$ in $\mathbb C^n$, we write
$$\langle z, w\rangle = z_1 \bar w_1 +\cdots+ z_n \bar w_n,$$ and $$|z|^2 = |z_1 |^2 +\cdots + |z_n |^2.$$
For any real vector ${\alpha}=(\alpha_1,\cdots,\alpha_n)$, where each $\alpha_j$ satisfies $\alpha_j>-1$, in which case we write $\alpha>-{\bf 1}$, we consider the measure on $\mathbb D^n$ given by $$dv_{ {\alpha}}(z)=c_{ \alpha}\prod_{j=1}^n(1-|z_j|^2)^{\alpha_j}d\nu(z_j),$$
$c_{ {\alpha}}=\prod_{j=1}^n(\alpha_j+1)$, $d\nu$ being the Lebesgue area measure on $\mathbb D$, normalized so that the measure of $\mathbb D$ is $1$. Let $H(\mathbb D^n)$ denote the space of analytic functions on $\mathbb D^n$. For $0 < p < \infty $, the (vector) weighted Bergman space
$A_{ {\alpha}}^p(\mathbb D^n)$ is the intersection $L^p (\mathbb D^n , dv_{ {\alpha}} ) \cap H(\mathbb D^n )$.
The (unweighted) Bergman space $A^p$ corresponds to $ {\alpha}=(0,\cdots,0)$. We will be writing
$$\|f\|_{p, {\alpha}}:=\left(\int_{\mathbb D^n}|f (z)|^p dv_{ {\alpha}} (z)\right)^\frac{1}{p}
$$
and
$$\langle f,g\rangle_{ \alpha}:=\int_{\mathbb D^n}f (z) \overline{g(z)}dv_{ {\alpha}} (z).$$
The Bergman space $A_{ {\alpha}}^{2}(\mathbb D^n)$ is a reproducing kernel Hilbert space with its kernel $K^{ {\alpha}}$ given by
 $$K^{ {\alpha}}(z,w)=\prod_{j=1}^n\frac{1}{(1-z_j\bar w_j)^{\alpha_j+2}}.$$
 That is for any $f\in A_{ {\alpha}}^{2}(\mathbb D^n)$, the following representation holds
\begin{equation}
f(w)=P_{ {\alpha}}f(w)=\langle f,K_w^{ {\alpha}}\rangle_{ {\alpha}}=\int_{\mathbb D^n}f(z)K^{ {\alpha}}(w,z)dv_{ {\alpha}}(z).
\end{equation}
The positive Bergman projection $P_\alpha^+$ is defined by replacing $K_w^{ {\alpha}}$ by $|K_w^{ {\alpha}}|$ in the definition of $P_\alpha$. We remark that the boundedness of $P_\alpha^+$ implies the boundedness of $P_\alpha$.
\vskip .3cm
For $f\in L^2(\mathbb D^n, dv_{ {\alpha}})$, we can densely define on $A_{ {\alpha}}^{2}(\mathbb D^n)$ the Toeplitz operator $T_f$ with symbol $f$  as in (\ref{eq:toeplitzdef}).
\vskip .1cm
Very recently bounded and invertible Toeplitz products $T_fT_{\overline g}$ on $A^2(\mathbb D^n)$ have been characterized in \cite{LS}. We extend this characterization to $A_{ \alpha}^p(\mathbb D^n)$, $1<p<\infty$. More precisely, denoting $k_w^\alpha$ the normalized kernel in $A_\alpha^2(\mathbb D^n)$, we have the following generalization.
\begin{theorem}\label{theo:boundinvertpolydisc}
Let $\alpha> -{\bf 1}=(-1,-1,\cdots,-1)$, $1<p<\infty$, $pq=p+q$, and assume that $f\in A_\alpha^p(\mathbb D^n)$ and $g\in A_\alpha^q(\mathbb D^n)$. Then the following assertions hold.
\begin{itemize}
\item[(i)] If $T_fT_{\overline g}$ is bounded and invertible on $A_\alpha^p(\mathbb D^n)$, then
\Be\label{eq:Sarasonpolydisc}
[f,g]_{p,\alpha}:=\sup_{z\in \mathbb D^n}\|fk_w^\alpha\|_{p,\alpha}\|gk_w^\alpha\|_{q,\alpha}<\infty
\Ee
and this is equivalent to
\Be\label{eq:infcondpolydisc}
\eta:=\inf_{z\in \mathbb D^n}|f(z)||g(z)|>0.
\Ee

\item[(ii)] If both (\ref{eq:Sarasonpolydisc}) and (\ref{eq:infcondpolydisc}) hold, then $$[f]_{p,\alpha}:=\sup_{w\in \mathbb D^n}\|fk_w^\alpha\|_{p,\alpha}\|f^{-1}k_w^\alpha\|_{q,\alpha}<\eta^{-2}[f,g]_{p,\alpha}^2,$$ and $T_fT_{\overline g}$ is bounded and invertible on $A_\alpha^p(\mathbb D^n)$. Moreover,
    $$\|T_fT_{\overline g}\|\le C\eta^{-2n\times\max\{p,q\}}[f,g]_{2,\alpha}^{1+2n\times\max\{p,q\}}.$$
\end{itemize}
\end{theorem}

For the proof of the necessary part in Theorem \ref{theo:boundinvertpolydisc}, we will adapt the proof of the necessary condition (\ref{eq:sarasonp}) for the boundedness of the Toeplitz products in \cite{Miao}. For the sufficient part, we define product B\'ekoll\'e-Bonami weights and use the corresponding estimate of the Bergman projection of the polydisc.
\vskip .2cm
In the next section, we provide some useful results needed later. Theorem \ref{theo:boundinvertpolydisc} is proved in Section 3. 

As usual, given two positive quantities $A$ and $B$, the notation $A\lesssim B$ (resp. $A\gtrsim B$)  means that $A\le CB$ (resp. $B\le CA$) for some absolute positive constant $C$. The notation $A\backsimeq B$ mean that $A\lesssim B$ and $B\lesssim A$. We will use $C(p)$ to say that the constant $C$ depends only on $p$.

\section{Some useful tools}
We give in this section some useful facts needed in our proofs.
\subsection{Estimates and weighted estimates in one parameter}
We recall that the Carleson box associated to an interval $I\subset \mathbb T$ is defined by
\Be\label{eq:Carlboxdisc}
Q_I:\{re^{i\theta}:1-|I|<r<1,\,\,\textrm{and}\,\,\,e^{i\theta}\in I\}.
\Ee

The following is easy to check.
\begin{lemma}\label{lem:kernelpnorm}
Let $1<p<\infty$, $-1< \alpha<\infty$. Then
$$\|k_w^\alpha\|_{p,\alpha}\backsimeq (1-|w|^2)^{(\frac{2+\alpha}{2})(\frac{2}{p}-1)},\,\,\,w=x+iy\in \mathbb{D}.$$
\end{lemma}

We have the following estimate (see for example \cite{Isral}).
\begin{lemma}\label{lem:berezinestim}
Let $1<p<\infty$, $-1< \alpha<\infty$. Then there is a constant $C>0$ such that for any $f\in A_\alpha^p(\mathbb{D})$,
\begin{equation}\label{eq:berezinestim}
|f(z)|^p\le C\|k_z^\alpha\|_{p,\alpha}^{-p}\|fk_z^\alpha\|_{p,\alpha}^p,\,\,\,z\in \mathbb{D}.
\end{equation}
\end{lemma}

Let $\omega$ be a positive function defined on $\mathbb{D}$ and $\alpha>-1$. For $1<p<\infty$, we say $\omega$ is a B\'ekoll\'e-Bonami weight (or $\omega$ belongs to the class $B_{p,\alpha}(\mathbb{D})$ if
$$[\omega]_{B_{p,\alpha}}:=\sup_{I\subset \mathbb T,\,\,\, I\,\,\,\textrm{interval}}\left(\frac{1}{|I|^{2+\alpha}}\int_{Q_I}\omega(z)d\nu_\alpha(z)\right)\left(\frac{1}{|I|^{2+\alpha}}\int_{Q_I}\omega(z)^{1-q}d\nu_\alpha(z)\right)^{p-1}<\infty,$$
$pq=p+q$.
\vskip .3cm
D. B\'ekoll\'e and A. Bonami proved in \cite{Bek} and \cite{BB} that the Bergman projection $P_\alpha$ is bounded on $L^p(\mathbb{D},\omega d\nu_\alpha)$ if and only if the weight function $\omega$ is in the class $B_{p,\alpha}(\mathbb{D})$. The following estimate of the weighted norm of the Bergman projection is due to S. Pott and M. C. Reguera.
\begin{proposition}[Pott-Reguera \cite{PR}]\label{prop:pott}
Let $1<p,q<\infty$, $p=q(p-1)$ and $-1< \alpha<\infty$. Suppose that $\omega\in B_{p,\alpha}(\mathbb{D}) $. Then
$P_\alpha$ and $P_\alpha^+$ are bounded on $L^p(\omega d\nu_\alpha)$.
Moreover, $\|P_\alpha^+\|_{L^p(\omega d\nu_\alpha)\rightarrow L^p(\omega d\nu_\alpha)}\le C(p)[\omega]_{B_{p,\alpha}}^{\max\{1,\frac{q}{p}\}}$.
\end{proposition}

Let us observe the following easy to prove fact (see for example \cite{Isral}).
\begin{lemma}\label{lem:invtobekweight}
Let $-{ 1}< \alpha<{\infty}$. Assume that $1<p<\infty$ and put $p=q(p-1)$. If $f$ is analytic on $\mathbb{D}$ with $$[f]_{p,\alpha}:=\sup_{w\in \mathbb{D}^n}\|fk_w^\alpha\|_{p,\alpha}\|f^{-1}k_w^\alpha\|_{q,\alpha}<\infty,$$
then $\omega=|f|^p$ belongs to the class $B_{p,\alpha}(\mathbb{D})$. Moreover, $$[\omega]_{B_{p,\alpha}(\mathbb{D})}\lesssim [f]_{p,\alpha}^p.$$
\end{lemma}

\subsection{Multi-parameter estimates}
Given $\alpha=(\alpha_1,\cdots, \alpha_n)\in \mathbb R^n$ and $\beta=(\beta_1,\cdots, \beta_n)\in \mathbb R^n$, the notation $\alpha<\beta$ (resp. $\alpha=\beta$) means that $\alpha_j<\beta_j$ (resp. $\alpha_j=\beta_j$), $j=1,\cdots,n$. The notation $\alpha>\beta$ is equivalent to $\beta<\alpha$ and $\alpha\le \beta$ is equivalent to $\alpha<\beta$ or $\alpha=\beta$. We also use the notations ${\bf 0}=(0,\cdots,0)$, $-{\bf 1}=(-1,\cdots,-1)$, and $\underline {\infty}=(\infty,\cdots,\infty)$.

\vskip .2cm

Let us introduce the class $\mathcal {B}_{p,\alpha}(\mathbb{D}^n)$, $1<p<\infty$, $\mathbb R^n\ni \alpha=(\alpha_1,\cdots,\alpha_n)>-{\bf 1}$. We say a positive locally integrable function $\omega$ on $\mathbb{D}^n$ belongs to $\mathcal {B}_{p,\alpha}(\mathbb{D}^n)$ if there is a constant $C>0$ such that for any $k\in \{1,\cdots,n\}$,
\Be\label{eq:defprodbekbo}\sup_{\xi=(\xi_1,\cdots,\xi_{k-1},\xi_{k+1},\cdots,\xi_n)\in \mathbb{D}^{n-1}} [\omega(\xi_1,\cdots,\xi_{k-1},\cdot,\xi_{k+1},\cdots,\xi_n)]_{B_{p,\alpha_k}(\mathbb{D})}\le C.\Ee
We denote by $[\omega]_{\mathcal {B}_{p,\alpha}(\mathbb{D}^n)}$ the infinimum of the constants $C$ in (\ref{eq:defprodbekbo}).
Note that the class $\mathcal {B}_{p,\alpha}(\mathbb{D}^n)$ is not empty; one easily checks that the weight $\omega=\prod_{j=1}^n\omega_j$ where $\omega_j\in B_{p,\alpha_j}(\mathbb{D})$ belongs to $\mathcal {B}_{p,\alpha}(\mathbb{D}^n)$.

From our definition of the multi-parameter B\'ekoll\'e-Bonami classes and Proposition \ref{prop:pott} we easily deduce the following.
\begin{proposition}\label{prop:Berprojprodcont}
Let $1<p,q<\infty$, $p=q(p-1)$ and $-{\bf 1}< \alpha=(\alpha_1,\cdots,\alpha_n)<\underline {\infty}$. Suppose that $\omega\in \mathcal {B}_{p,\alpha}(\mathbb{D}^n)$. Then both $P_\alpha$ and $P_\alpha^+$ are bounded on $L^p(\mathbb{D}^n, \omega d\nu_{\alpha}(z))$. Moreover,
$$\|P_\alpha^+\|_{L^p(\mathbb{D}^n, \omega d\nu_{\alpha}(z))\rightarrow L^p(\mathbb{D}^n, \omega d\nu_{\alpha}(z))}\le C(p)[\omega]_{\mathcal {B}_{p,\alpha}(\mathbb{D}^n)}^{n\times {\max\{1,\frac{q}{p}\}}}.$$
\end{proposition}

Let us also observe the following.
\begin{lemma}\label{lem:prodinvarcond}
Let $-{\bf 1}< \alpha<\underline {\infty}$. Assume that $1<p<\infty$ and put $p=q(p-1)$. If $f$ is analytic on $\mathbb{D}^n$ with $$[f]_{p,\alpha}:=\sup_{w\in \mathbb{D}^n}\|fk_w^\alpha\|_{p,\alpha}\|f^{-1}k_w^\alpha\|_{q,\alpha}<\infty,$$ then there is a constant $C>0$ such that for any $1\le k\le n$, and any $w=(w_1,\cdots,w_{k-1},w_{k+1},\cdots,w_{n-1})\in \mathbb{D}^{n-1}$,
$$\sup_{z\in \mathbb{D}}\|f_wk_{z}^{\alpha_k}\|_{p,\alpha_k}\|f_w^{-1}k_z^{\alpha_k}\|_{q,\alpha_k}\le C\sup_{\xi\in \mathbb{D}^n}\|fk_\xi^\alpha\|_{p,\alpha}\|f^{-1}k_\xi^\alpha\|_{q,\alpha},$$
where $f_w(z)=f(\zeta_1,\cdots,\zeta_n)$, $\zeta_j=w_j$ for $j\neq k$ and $\zeta_k=z$.
\end{lemma}
\begin{proof}
We may suppose that $k=1$. Let $w=(w_1,\cdots,w_{n-1})\in \mathbb{D}^{n-1}$ be given. For any $z=x+iy\in \mathbb{D}$ given, we put $\zeta=(z,w_1,\cdots, w_{n-1})$ and $\tilde {\alpha}=(\alpha_2,\cdots,\alpha_n)$. Then using Lemma \ref{lem:kernelpnorm} and Lemma \ref{lem:berezinestim}, we obtain
\Beas
\|f_wk_{z}^{\alpha_1}\|_{p,\alpha_1}^p &=& \int_{\mathbb{D}}|f(\xi,w)|^p|k_z^{\alpha_1}(\xi)|^p(1-|\xi|^2)^{\alpha_1}d\nu(\xi)\\ &=& \int_{\mathbb{D}}|f_{\xi}(w)|^p|k_z^{\alpha_1}(\xi)|^p(1-|\xi|^2)^{\alpha_1}d\nu(\xi)\\ &\lesssim& \|k_w^{\tilde {\alpha}}\|_{p,\tilde {\alpha}}^{-p}\int_{\mathbb{D}}\|f_{\xi}k_w^{\tilde {\alpha}}\|_{p,\tilde{\alpha}}^p|k_z^{\alpha_1}(\xi)|^p(1-|\xi|^2)^{\alpha_1}d\nu(\xi)\\ &=&  \|k_w^{\tilde {\alpha}}\|_{p,\tilde {\alpha}}^{-p}\|fk_{\zeta}^\alpha\|_{p,\alpha}^p.
\Eeas
In the same way, we obtain
$$\|f_w^{-1}k_z^{\alpha_1}\|_{q,\alpha_1}\lesssim \|k_w^{\tilde {\alpha}}\|_{q,\tilde {\alpha}}^{-q}\|f^{-1}k_{\zeta}^\alpha\|_{q,\alpha}^q.$$
Thus as $\|k_w^{\tilde {\alpha}}\|_{p,\tilde {\alpha}}^{-p}\|k_w^{\tilde {\alpha}}\|_{q,\tilde {\alpha}}^{-q}\backsimeq 1$, we obtain for any $k\in \{1,2,\cdots,n\}$,
$$\|f_wk_{z}^{\alpha_k}\|_{p,\alpha_k}\|f_w^{-1}k_z^{\alpha_k}\|_{q,\alpha_k}\lesssim \sup_{\zeta\in \mathbb{D}^n}\|fk_\zeta^\alpha\|_{p,\alpha}\|f^{-1}k_\zeta^\alpha\|_{q,\alpha}.$$
The proof is complete.
\end{proof}
Putting together the definition of B\'ekoll\'e-Bonami classes of the polydisc, the above lemma and Lemma \ref{lem:invtobekweight}, we obtain the following.
\begin{lemma}\label{lem:prodinvtobekweight}
Let $-{\bf 1}< \alpha<\underline {\infty}$. Assume that $1<p<\infty$ and put $p=q(p-1)$. If $f$ is an analytic function on $\mathbb{D}^n$ such that $$[f]_{p,\alpha}:=\sup_{w\in \mathbb{D}^n}\|fk_w^\alpha\|_{p,\alpha}\|f^{-1}k_w^\alpha\|_{q,\alpha}<\infty,$$ then the function $\omega=|f|^p$ belongs to the class $\mathcal {B}_{p,\alpha}(\mathbb{D}^n)$. Moreover, $$[\omega]_{B_{p,\alpha}(\mathbb{D}^n)}\lesssim [f]_{p,\alpha}^p.$$
\end{lemma}

\section{Bounded and invertible Toeplitz products on the unit Polydisc of $\mathbb C^n$}
Before proving Theorem \ref{theo:boundinvertpolydisc}, let us prove some useful results. We start by the following lemma.
\begin{lemma}\label{lem:lowerboundofprod}
Let $\alpha>-{\bf 1}$, $1<p<\infty$, $qp=p+q$, and let $f\in A_\alpha^p(\mathbb D^n)$, $g\in A_\alpha^q(\mathbb D^n)$. Then
\Be\label{eq:lowerboundofprod}
\|f\|_{p,\alpha}\|g\|_{q,\alpha}\lesssim \|T_fT_{\overline g}\|.
\Ee
\end{lemma}
\begin{proof}
Let $f\in A_\alpha^p(\mathbb D^n)$ and $g\in A_\alpha^q(\mathbb D^n)$. Consider on $A_{\alpha}^p (\mathbb D^n )$ the operator $f\otimes g$  defined by
$$(f \otimes g)h = \langle h, g\rangle_{\alpha} f,$$
for $h\in A_{\alpha}^p(\mathbb D^n)$. It is easily proved that $f \otimes g$ is bounded on $A_{\alpha}^p(\mathbb D^n)$ with norm equals to
\Be\label{eq:otimesnorm}
\|f \otimes g\|=\|f\|_{p,\alpha}\|g\|_{q,\alpha}.
\Ee
\vskip .1cm
Next recall that the Berezin transform of an operator $S$ defined on $A_{\alpha}^2(\mathbb D^n)$ is given by
\begin{equation}\label{eq:Berezindef}
B_{\alpha}(S) (w)=\langle Sk_w^{\alpha},k_w^{\alpha}\rangle_{\alpha}.
\end{equation}
Observing that
$T_{\overline f}K^{\alpha}_w=\overline {f}(w)K^{\alpha}_w$, we easily obtain
\begin{equation}\label{eq:berenzinofprodtoep}
B_{\alpha}(T_fT_{\overline g})(w)=f(w)\overline {g}(w).
\end{equation}
One also proves as in the one parameter case that

\begin{equation}\label{eq:berenzinofrank1oper}
B_{\alpha}(f \otimes g)(w)=\langle (f \otimes g)k_w^{\alpha},k_w^{\alpha}\rangle_{\alpha}=\prod_{j=1}^n{(1-|w_j|^2)^{\alpha_j+2}}f(w)\overline{g(w)}.
\end{equation}

Note also that if ${(\lambda)_{\alpha_0,k}}$ is a sequence defined so that the Taylor expansion of $(1-x)^{2+\alpha_0}$ in a neighborhood of the origin
 is \Be\label{eq:Taylorparaone}(1-x)^{2+\alpha_0}=\sum_{k=0}^\infty \lambda_{\alpha_0,k}x^k,\Ee then for $\alpha_0>-1$ the series in the right hand side of (\ref{eq:Taylorparaone}) is absolutely convergent in the closed unit disc of the complex plane. Hence, if $x=(x_1,\cdots,x_n)$ and $\lambda_{\alpha,l}=\lambda_{\alpha_1,l_1}\cdots \lambda_{\alpha_n,l_n}$ we have
\begin{equation}\label{eq:taylor}
\prod_{j=1}^n(1-x_j)^{2+\alpha_j}=\sum_{{ l}\in \mathbb N^n}\lambda_{\alpha,l}x^l,
\end{equation}
where $x^l=x_1^{l_1}\cdots x_n^{l_n}$ and the series in (\ref{eq:taylor}) is absolutely convergent in the closed unit polydisc.

It follows from (\ref{eq:berenzinofprodtoep}) that for $t=(t_1,\cdots,t_n)$, the Berezin transform of the product $T_{z^t}T_fT_{\overline g}T_{\overline {z}^t}$ is given by
\Beas
B_\alpha\left(T_{z^t}T_fT_{\overline g}T_{\overline {z}^t}\right)(w) &=& B_\alpha\left(T_{z^t f}T_{\overline {gz^t}}\right)(w)\\
&=& w^t f(w)\overline {g(w)w^t}.
\Eeas
Thus using the expansion (\ref{eq:taylor}), equation (\ref{eq:berenzinofrank1oper}) and the injectivity of the Berezin transform, we obtain
\begin{equation}\label{eq:rank1toepl}
f \otimes g = \sum_{l\in \mathbb N^n}\lambda_{\alpha,l}T_{z^l}T_fT_{\overline g}T_{\overline {z^l}}=\sum_{l\in \mathbb N^n}\lambda_{\alpha,l}T_{z_1^{l_1}\cdots z_n^{l_n}}T_fT_{\overline g}T_{\overline {z_1^{l_1}\cdots z_n^{l_n}}}.
\end{equation}
It follows that if $s_{\alpha}=\sum_{l\in \mathbb N^n}|\lambda_{\alpha,l}|$, then as $||T_z||=1$, we have
\begin{equation*}
||f||_{p,\alpha}||g||_{q,\alpha}=||f \otimes g||\le s_{\alpha}||T_fT_{\overline g}||
\end{equation*}
for $f\in A_{\alpha}^p(\mathbb D^n)$ and $g\in A_{\alpha}^q(\mathbb D^n)$. The proof is complete.
\end{proof}
We also need the following necessary condition for the boundedeness of the Toeplitz product. It extends the one-parameter result in \cite{Miao}.
\begin{proposition}\label{prop:miaopolydisc}
Let $\alpha>-{\bf 1}$, $1<p<\infty$, $qp=p+q$, and let $f\in A_\alpha^p(\mathbb D^n)$, $g\in A_\alpha^q(\mathbb D^n)$. If $T_fT_{\overline g}$ is bounded on $A_\alpha^p(\mathbb D^n)$, then (\ref{eq:Sarasonpolydisc}) holds. Moreover,
\Be\label{eq:necessboundpoly}
[f,g]_{p,\alpha}\lesssim \|T_fT_{\overline g}\|.
\Ee
\end{proposition}
\begin{proof} The proof is essentially the same as in the one-parameter in \cite{Miao}. Let us give it here for completeness. We start by recalling that for $a\in \mathbb D$, the transform $z\mapsto \varphi_a(z)=\frac{a-z}{1-\overline {a}z}$ is an automorphism of the
unit disc $\mathbb D$ with inverse $\varphi_a^{-1}=\varphi_a$, and such that $\varphi_a(0)=a$ and $\varphi_a(a)=0$. Now for $a=(a_1,\cdots,a_n)\in \mathbb D^n$, the corresponding mapping in the unit polydisc $\mathbb D^n$ is the transform $z=(z_1,\cdots,z_n)\in \mathbb {D}^n\mapsto \varphi_a(z)=(\varphi_{a_1}(z_1),\cdots, \varphi_{a_n}(z_n))$. This is an automorphism of $\mathbb D^n$ and its real Jacobian is $\prod_{j=1}^n|\varphi_{a_j}'(z_j)|^2$ (see \cite{SZ2}). It follows that as in the one-parameter case, the change of variables $w = \varphi_a(z)$ works in the polydisc and the following holds:
\begin{equation}\label{eq:changevar}
\int_{\mathbb D^n}f(\varphi_a(z))dv_{\alpha}(z)=\int_{\mathbb D^n}f(z)|k_a^{\alpha}(z)|^2dv_{\alpha}(z).
\end{equation}
Let us consider the mapping $U_a(h)=\left(h\circ \varphi_a\right)k_a^{\alpha}$. It is easy to check that $U_a$ is an isometry of $A_{\alpha}^2(\mathbb D^n)$. Also, $U_a$ is unitary and commutes with the Toeplitz operators in the following sense:
\begin{equation}\label{eq:UTchange}
U_aT_f=T_{f\circ \varphi_a}U_a,\,\,\,f\in A_{\alpha}^2(\mathbb D^n)
\end{equation}
and consequently, $$T_{\overline g}U_a=U_aT_{\overline {g\circ \varphi_a}}$$(see \cite{SZ2}). As a consequence, one obtains that
for any $f_1\in A_\alpha^p(\mathbb D^n)$ and $g_1\in A_\alpha^q(\mathbb D^n)$,
\Be\label{eq:miao1}
T_{f_1\circ \varphi_a}T_{\overline {g_1}\circ \varphi_a} = U_a\left(T_{f_1}T_{\overline {g_1}}\right)U_a.
\Ee
As said above, $U_a$ is  an isometry of $A_{\alpha}^2(\mathbb D^n)$. Using (\ref{eq:changevar}), one checks that the following operator also provides an isometry of $A_{\alpha}^p(\mathbb D^n)$,
$$U_a^p(h)=\left(h\circ \varphi_a\right)k_a^{\alpha}\left(\overline {k_a^{\alpha}}\right)^{2/p-1}.$$
The following operator was introduced in one-parameter case by J. Miao in \cite{Miao}:
$$ V_a^p(h)=P_\alpha\left(U_a^p(h)\right) =P_\alpha\left(\left(h\circ \varphi_a\right)k_a^{\alpha}\left(\overline {k_a^{\alpha}}\right)^{2/p-1}\right).$$
Using the following identity which holds for any $f_2\in A_{\alpha}^p(\mathbb D^n)$ and any $g_2\in A_{\alpha}^q(\mathbb D^n)$,
\Be\label{eq:miao2}
T_{\overline {f_2}}P_\alpha(g_2)=P_\alpha(\overline {f_2}g_2),
\Ee
we obtain for any $h\in A_{\alpha}^p(\mathbb D^n)$,
\Beas
T_{\left(\overline {k_a^{\alpha}}\right)^{1-2/p}}V_a^p(h) &=& P_\alpha\left(\left(\overline {k_a^{\alpha}}\right)^{1-2/p}\left(h\circ \varphi_a\right)k_a^{\alpha}\left(\overline {k_a^{\alpha}}\right)^{2/p-1}\right)\\ &=& \left(h\circ \varphi_a\right)k_a^{\alpha}.
\Eeas
That is
\Be\label{eq:miao3}
T_{\left(\overline {k_a^{\alpha}}\right)^{1-2/p}}V_a^p(h)=U_a(h).
\Ee
From (\ref{eq:miao1}), we have for any $u\in A_{\alpha}^p(\mathbb D^n)$ and any $v\in A_{\alpha}^q(\mathbb D^n)$,
\Be\label{eq:miao01} \langle T_{f_1\circ \varphi_a}T_{\overline {g_1}\circ \varphi_a}u,v\rangle_\alpha = \langle T_{\overline {g_1}}\left(U_a(u)\right), T_{\overline {f_1}}\left(U_a(v)\right)\rangle_\alpha.
\Ee
It follows from (\ref{eq:miao01}) and (\ref{eq:miao3}) that for any $u\in A_{\alpha}^p(\mathbb D^n)$ and any $v\in A_{\alpha}^q(\mathbb D^n)$,
\Be\label{eq:miao4}
\langle T_{f_1\circ \varphi_a}T_{\overline {g_1}\circ \varphi_a}u,v\rangle_\alpha=\langle T_{\overline {g_1}}T_{\left(\overline {k_a^{\alpha}}\right)^{1-2/p}}V_a^p(u), T_{\overline {f_1}}T_{\left(\overline {k_a^{\alpha}}\right)^{1-2/q}}V_a^q(v)\rangle_\alpha.
\Ee
Let us take
$$f_1=\frac{f}{\left(k_a^{\alpha}\right)^{1-2/q}},\,\,\,g_1=\frac{g}{\left(k_a^{\alpha}\right)^{1-2/p}}.$$
Clearly, $f_1$ and $g_1$ are holomorphic and consequently,
$$T_{\overline {f_1}}T_{\left(\overline {k_a^{\alpha}}\right)^{1-2/q}}=T_{\overline {f_1\left(k_a^{\alpha}\right)^{1-2/q}}}=T_{\bar f}$$
and
$$T_{\overline {g_1}}T_{\left(\overline {k_a^{\alpha}}\right)^{1-2/p}}=T_{\overline {g_1\left(k_a^{\alpha}\right)^{1-2/p}}}=T_{\bar g}.$$
It follows from the H\"older's inequality that
\Beas
\left|\langle T_{f_1\circ \varphi_a}T_{\overline {g_1}\circ \varphi_a}u,v\rangle_\alpha\right| &=& \left|\langle T_{f}T_{\overline {g}}V_a^p(u),V_a^q(v)\rangle_\alpha\right|\\ &\le& \|T_{f}T_{\overline {g}}\|\|V_a^p(u)\|_{p,\alpha}\|V_a^q(v)\|_{q,\alpha}\\ &\le& \|T_{f}T_{\overline {g}}\|\|U_a^p(u)\|_{p,\alpha}\|U_a^q(v)\|_{q,\alpha}\\ &=& \|T_{f}T_{\overline {g}}\|\|u\|_{p,\alpha}\|v\|_{q,\alpha}
\Eeas
Thus using Lemma \ref{lem:lowerboundofprod},
$$\|f_1\circ \varphi_a\|_{p,\alpha}\|g_1\circ \varphi_a\|_{q,\alpha}\lesssim \| T_{f_1\circ \varphi_a}T_{\overline {g_1}\circ \varphi_a}\|\le \|T_{f}T_{\overline {g}}\|.$$
Next we observe that
$$
f_1\circ \varphi_a = (f\circ \varphi_a)(k_a^\alpha\circ \varphi_a)^{-1+\frac{2}{q}}=(f\circ \varphi_a)(k_a^\alpha\circ \varphi_a)^{1-\frac{2}{p}}
$$
and consequently using (\ref{eq:changevar}) that
$$\|f_1\circ \varphi_a\|_{p,\alpha}=\|fk_a^\alpha\|_{p,\alpha}.$$
In the same way, we obtain
$$\|g_1\circ \varphi_a\|_{q,\alpha}=\|gk_a^\alpha\|_{q,\alpha}.$$
We conclude that
$$\|fk_a^\alpha\|_{p,\alpha}\|gk_a^\alpha\|_{q,\alpha}\lesssim \|T_{f}T_{\overline {g}}\|.$$
The proof is complete.
\end{proof}
Let us now prove the following.
\begin{proposition}\label{prop:main11}
Let $1<p,q<\infty$, $p=q(p-1)$ and $-1< \alpha<\infty$. Given $f\in A_\alpha^p(\mathbb{D}^n)$ and $g\in A_\alpha^q(\mathbb{D}^n)$, if the product $T_fT_{\overline g}$ is bounded and invertible on $A_\alpha^p(\mathbb{D}^n)$, then the two following conditions are equivalent.
\begin{itemize}
\item[({\it i})]
\begin{equation}\label{eq:invertplane11}
\inf_{w\in \mathbb{D}^n}|f(w)||g(w)|>0.
\end{equation}
\item[({\it ii})]
\begin{equation}\label{eq:contradsaraplane11}
\sup_{w\in \mathbb{D}^n}\|fk_w^\alpha\|_{p,\alpha}\|gk_w^\alpha\|_{q,\alpha}<\infty.
\end{equation}
\end{itemize}
\end{proposition}
\begin{proof}
We assume that $T_fT_{\overline g}$ is bounded and invertible on $A_\alpha^p(\mathbb{D}^n)$, so that its adjoint $\left(T_fT_{\overline g}\right)^*=T_gT_{\overline f}$ is bounded and invertible on $A_\alpha^q(\mathbb{D}^n)$.
Let us first observe  that the following conditions obviously hold. 
\begin{equation}\label{eq:invertplane21}
M_1:=\sup_{w\in \mathbb{D}^n}\|k_w^\alpha\|_{p,\alpha}^{-1}\|T_fT_{\overline g}k_w^\alpha\|_{p,\alpha}<\infty
\end{equation}

and
\begin{equation}\label{eq:invertplane3}
M_2:=\sup_{w\in \mathbb{D}^n}\|k_w^\alpha\|_{q,\alpha}^{-1}\|T_gT_{\overline f}k_w^\alpha\|_{q,\alpha}<\infty.
\end{equation}

Next, observing that given $w\in \mathbb{D}^n$, $T_fT_{\overline g}k_w^\alpha=\overline {g(w)}fk_w^\alpha$, we obtain that for any $w\in \mathbb{D}^n$,
\Beas 1 &\le&  \|k_w^\alpha\|_{p,\alpha}^{-1}\|\left(T_fT_{\overline g}\right)^{-1}\left(\overline {g(w)}fk_w^\alpha\right)\|_{p,\alpha}\\ 
&\le& \|k_w^\alpha\|_{p,\alpha}^{-1}\|\left(T_fT_{\overline g}\right)^{-1}\||g(w)|\|fk_w^\alpha\|_{p,\alpha}.
\Eeas
That is
\begin{equation}\label{eq:ineq1}\|k_w^\alpha\|_{p,\alpha}^{-1}|g(w)|\|fk_w^\alpha\|_{p,\alpha}\ge \frac{1}{\|\left(T_fT_{\overline g}\right)^{-1}\|}
\end{equation}

We also have that
\begin{eqnarray*} 1 &\le&  \|k_w^\alpha\|_{q,\alpha}^{-1}\|\left(T_gT_{\overline f}\right)^{-1}\left(\overline {f(w)}gk_w^\alpha\right)\|_{q, \alpha}\\ 
&\le& \|k_w^\alpha\|_{q,\alpha}^{-1}\|\left(T_gT_{\overline f}\right)^{-1}\||f(w)|\|gk_w^\alpha\|_{q,\alpha}.
\end{eqnarray*}
Which provides that
\begin{equation}\label{eq:ineq2}\|k_w^\alpha\|_{q,\alpha}^{-1}|f(w)|\|gk_w^\alpha\|_{q,\alpha}\ge \frac{1}{\|\left(T_gT_{\overline f}\right)^{-1}\|}.
\end{equation}
Combining (\ref{eq:ineq1}) and (\ref{eq:ineq2}), we see that
there are two constants $c_1, c_2>0$ such that for any $w\in \mathbb{D}^n$,
\begin{equation}\label{eq:infprodineq}
\frac{c_1}{\|\left(T_fT_{\overline g}\right)^{-1}\|\|\left(T_gT_{\overline f}\right)^{-1}\|}\le |f(w)||g(w)|\|gk_w^\alpha\|_{q,\alpha}\|fk_w^\alpha\|_{p,\alpha}\le c_2M_1M_2.
\end{equation}
Now suppose that $$M=\sup_{w\in \mathbb{D}}\|gk_w^\alpha\|_{q,\alpha}\|fk_w^\alpha\|_{p,\alpha}<\infty.$$
Then from the left inequality in (\ref{eq:infprodineq}) we get that for any $w\in \mathbb{D}^n,$
$$|f(w)||g(w)|\ge \frac{c_1}{M\|\left(T_fT_{\overline g}\right)^{-1}\|\|\left(T_gT_{\overline f}\right)^{-1}\|}>0.$$
Hence (\ref{eq:invertplane11}) holds.

Conversely, if $\eta=\inf_{w\in \mathbb{D}^n}|f(w)||g(w)|>0$, then the right inequality in (\ref{eq:infprodineq}) provides
$$\|gk_w^\alpha\|_{q,\alpha}\|fk_w^\alpha\|_{p,\alpha}\le \frac{c_2M_1M_2}{\eta}<\infty$$
for any $w\in \mathbb{D}$. That is (\ref{eq:contradsaraplane11}) holds. The proof is complete.
\end{proof}

\begin{proof}[Proof of Theorem \ref{theo:boundinvertpolydisc}]
Note that if $T_fT_{\overline g}$ is bounded on $A_\alpha^p(\mathbb D^n)$, then (\ref{eq:Sarasonpolydisc}) holds by Proposition \ref{prop:miaopolydisc}. Hence if $T_fT_{\overline g}$ is bounded and invertible on $A_\alpha^p(\mathbb D^n)$, (\ref{eq:Sarasonpolydisc}) holds and by Proposition \ref{prop:main11}  this is equivalent to (\ref{eq:infcondpolydisc}).

Now let us suppose that both (\ref{eq:Sarasonpolydisc}) and (\ref{eq:infcondpolydisc}) are satisfied. We first prove that the weight $\omega=|f|^p$ belongs to the class $\mathcal {B}_{p,\alpha}(\mathbb D^n)$. Let us observe for this that as $\eta=\inf_{z\in \mathbb D^n}|f(z)||g(z)|>0$, we have that for any $z\in \mathbb D^n$,
$|f(z)||g(z)|>\eta$ and consequence, that $|g(z)|>\eta|f(z)|^{-1}$. It follows using the analogue of the estimate (\ref{eq:berezinestim}) for the polydisc that for any $w\in \mathbb D^n$,
\Beas
[f]_{p,\alpha}:=\|fk_w^\alpha\|_{p,\alpha}\|f^{-1}k_w^\alpha\|_{q,\alpha} &=& \|fk_w^\alpha\|_{p,\alpha}|f(w)||f(w)|^{-1}\|f^{-1}k_w^\alpha\|_{q,\alpha}\\ &\le& \eta^{-2}\|fk_w^\alpha\|_{p,\alpha}|f(w)||g(w)|\|gk_w^\alpha\|_{q,\alpha}|\\ &\lesssim& \eta^{-2}\left(\|fk_w^\alpha\|_{p,\alpha}\|gk_w^\alpha\|_{q,\alpha}\right)^2\\ &\le& \eta^{-2}[f,g]_{p,\alpha}^2.
\Eeas
That is $$\sup_{w\in \mathbb D^n}\|fk_w^\alpha\|_{p,\alpha}\|f^{-1}k_w^\alpha\|_{q,\alpha}<\eta^{-2}[f,g]_{p,\alpha}^2$$
and from Lemma \ref{lem:prodinvtobekweight}, we have that this implies that $\omega=|f|^p$ belongs to the class $\mathcal {B}_{p,\alpha}(\mathbb D^n)$. Hence it follows from Proposition \ref{prop:Berprojprodcont}  that $P_\alpha^+$ is bounded on $L^p(\mathbb{D}^n, |f|^pd\nu_\alpha)$. It follows using Lemma \ref{lem:berezinestim} that
\Beas
|T_fT_{\overline g}h(z)| &=& |f(z)|\left|P_\alpha(\overline {g}h)(z)\right|\\ &\le& |f(z)|\int_{\mathbb{D}^n}|g(w)||h(w)||K^\alpha(z,w)|d\nu_\alpha(w)\\ &=& |f(z)|\int_{\mathbb{D}^n}|g(w)||f(w)||f(w)|^{-1}|h(w)||K^\alpha(z,w)|d\nu_\alpha(w)\\ &\lesssim& [f,g]_{p,\alpha}|f(z)|\int_{\mathbb{D}^n}|f(w)|^{-1}|h(w)||K^\alpha(z,w)|d\nu_\alpha(w)\\ &=& [f,g]_{p,\alpha}|f(z)|P_\alpha^+(|f|^{-1}|h|)(z).
\Eeas
Hence the boundedness of $T_fT_{\overline g}$ on $A_\alpha^p(\mathcal H)$ follows from the boundedness of $P_\alpha^+$ is bounded on $L^p(\mathbb{D}^n, |f|^pd\nu_\alpha)$  with the right estimate. Thus
$$\|T_fT_{\overline g}\|\le C(p)[\omega]_{B_{p,\alpha}(\mathbb{D}^n)}^{n\times \max\{1,\frac{q}{p}\}}[f,g]_{p,\alpha}\le C(p)[f,g]_{p,\alpha}[f]_{p,\alpha}^{n\times \max\{p,q\}}.$$
We conclude that
$$\|T_fT_{\overline g}\|\le C \eta^{-2n\times\max\{p,q\}}[f,g]_{p,\alpha}^{1+2n\times\max\{p,q\}}.$$

To prove that $T_fT_{\overline g}$ is invertible, we first observe with the right hand inequality in (\ref{eq:infprodineq}) that there is a constant $c_1>0$ such that for any $w\in \mathbb D^n$,
\Beas
c_1\le |f(w)||g(w)|\|fk_w^\alpha\|_{p,\alpha}\|gk_w^\alpha\|_{q,\alpha}\le |f(w)||g(w)|[f,g]_{p,\alpha}.
\Eeas
Hence that the function $(f\bar {g})^{-1}$ is bounded on $\mathbb D^n$. Thus the Toeplitz operator $T_{(f\bar g)^{-1}}$ is bounded on $A_\alpha^p(\mathbb D^n)$. Hence, denoting $I$ the identity operator, it follows that
$$T_fT_{\overline g}T_{(f\bar g)^{-1}}=I=T_{(f\bar g)^{-1}}T_fT_{\overline g}.$$
The proof is complete.
\end{proof}

We have the following particular case of Theorem \ref{theo:boundinvertpolydisc}.
\begin{corollary}\label{cor:boundinvertpolydisc}
Let $\alpha> -{\bf 1}=(-1,-1,\cdots,-1)$, $1<p<\infty$, $pq=p+q$, and assume that $f\in A_\alpha^p(\mathbb D^n)$ and $\frac{1}{f}\in A_\alpha^q(\mathbb D^n)$. Then $T_fT_{\overline g}$ is bounded and invertible on $A_\alpha^p(\mathbb D^n)$ if and only if $$[f]_{p,\alpha}:=\sup_{w\in \mathbb D^n}\|fk_w^\alpha\|_{p,\alpha}\|f^{-1}k_w^\alpha\|_{q,\alpha}<\infty.$$
\end{corollary}

\bibliographystyle{elsarticle-num}

\end{document}